\documentclass[12pt,english]{amsart}
\usepackage[cp1251]{inputenc}
\usepackage[english]{babel}
\usepackage{amsmath,amsthm,amssymb}
\newtheorem{theorem}{Theorem}[section]
\newtheorem{lemma}[theorem]{Lemma}
\newtheorem{corollary}[theorem]{Corollary}
\newtheorem{remark}[theorem]{Remark}

\theoremstyle{definition}
\newtheorem{definition}[theorem]{Definition}

\numberwithin{equation}{section}

\begin{document}

\title[On fractional regularity of distributions]{On fractional regularity of
distributions of functions in Gaussian random variables}

\author{Egor D. Kosov}

\maketitle

\begin{abstract}
We study fractional smoothness
of measures on $\mathbb{R}^k$, that are images
of a Gaussian measure under mappings from
Gaussian Sobolev classes.
As a consequence we obtain Nikolskii--Besov fractional regularity
of these distributions under some weak nondegeneracy assumption.
\end{abstract}



\section*{Introduction}
Let $\gamma$ be a Gaussian measure on a locally convex space $E$
and $f\colon E\to \mathbb{R}^k$ be a polynomial mapping.
It was shown in \cite{BKZ} and \cite{Kos} that the density of
the image measure $\gamma\circ f^{-1}$
belongs to a certain
Nikolskii--Besov class.
Here
we consider a general Sobolev mapping $f\in W^{p,2}(\gamma)$
and provide an estimate of the total variation norm
$\|(\gamma\circ f^{-1})_h - \gamma\circ f^{-1}\|_{\rm TV}$
in terms of the behavior of $\gamma(\Delta_f\le t)$ (see Theorems \ref{t3.1}, \ref{t4.1}
and Corollaries \ref{c3.1}, \ref{c4.1}),
where $\mu_h(A):=\mu(A-h)$ is the shift of the measure $\mu$ to the vector~$h$,
and where $\Delta_f$ is the determinant of the Malliavin matrix $M_f$ of the mapping $f$
(all the necessary definitions are given in the first section).
This result provides a quantitative estimate of smoothness of $\gamma\circ f^{-1}$
and complements
the classical theorem
(see \cite[Theorem 9.2.4]{DM}) which asserts that such a distribution
possesses a density with respect to the standard Lebesgue measure
if $\Delta_f(x)\ne0$ for $\gamma$-almost every point~$x$.
However, it should be mentioned that in this classical result
only the inclusion of $f$ to the first Sobolev class is assumed.
We also note that in \cite[Theorem 2.11]{BC} the lower semi-continuity of densities
of such distributions was established.

The obtained results also provide a quantitative estimate in the following
qualitative theorem (see \cite{BZ} and \cite{BKZ}, which generalizes \cite[Theorem 2.14]{BC}).
Let
$f_n=(f_{n,1},\ldots,f_{n,k})\colon\, E\to \mathbb{R}^k$
be a sequence of functions such that $f_{n,i}\in W^{4k,2}(\gamma)$.
Set
$$
\delta(\varepsilon):= \sup\limits_n \ \gamma(\Delta_{f_n}\le \varepsilon)
$$
and
assume that
$$
\sup_n \|f_n\|_{W^{4k,2}(\gamma)}=a<\infty \quad \hbox{and}\quad \lim\limits_{\varepsilon \to 0}\delta (\varepsilon) = 0.
$$
If the sequence of measures $\gamma\circ f_n^{-1}$ converges in
distribution, it also converges in variation.
Corollary \ref{c4.2} of the present paper
asserts that under the same assumptions one has
$$
\|\gamma\circ f_n^{-1} - \nu\|_{\rm TV}
\le C(k,a)\Bigl(\bigl[\delta\bigl(\|\gamma\circ f_n^{-1} - \nu\|_{\rm KR}
^{1/8}\bigr)\bigr]^{1/(4k)}  +
\|\gamma\circ f_n^{-1} - \nu\|_{\rm KR}^{1/(32k)}\Bigr),
$$
where $\nu$ is the limiting distribution and $\|\cdot\|_{\rm KR}$ is the
Kantorovich--Rubinstein norm, which metrizes weak convergence of probability measures.
A similar bound is also valid for mappings from $W^{p,2}(\gamma)$ for any $p>4k-1$,
which is also an improvement of the above result.

The approach in this work is similar to the classical Malliavin method
developed in \cite{Malliavin} (see also \cite{DM}).
The main idea of the method is to obtain bounds of the form
\begin{equation*}
\int \varphi^{(n)}(f)d\gamma \le C_n\sup_t|\varphi(t)|, \quad \forall \varphi\in C_0^\infty(\mathbb{R})
\end{equation*}
which yields that the density of $\gamma\circ f^{-1}$ is infinitely differentiable.
In works \cite{BKZ}, \cite{Kos}, the Malliavin condition was modified
to treat the case of Nikolskii--Besov fractional smoothness of distributions.
In this work we similarly employ the results of~\cite{KosBes}
which estimate the quantity $\|\mu_h - \mu\|_{\rm TV}$
in terms of the function
$$
\sigma(\mu, t):= \sup\Bigl\{\int\partial_e\varphi\, d\mu:\, \|\varphi\|_\infty\le t,\, \|\partial_e\varphi\|_\infty\le1\Bigr\},
$$
where the supremum is taken over all functions $\varphi\in C_0^\infty(\mathbb{R}^k)$ and unit vectors $e$.

To apply the classical Malliavin method one should assume some nondegeneracy of mapping~$f$,
for example in the form of integrability of $\Delta_f^{-1}$ to some power $p>1$.
Such condition is sometimes very restrictive and difficult for verification.
For example, the required integrability is not valid for polynomial mappings.
Nevertheless,
for polynomials on Gaussian space, the following weak nondegeneracy condition holds:
$\Delta_f^{-1}$ is integrable to every power $\theta<\frac{1}{2d(k-1)}$
(this follows from the Carbery--Wright inequality \cite{CarWr}, \cite{NSV}).
Thus, a natural question is to investigate the
smoothness properties of distributions $\gamma\circ f^{-1}$
for Sobolev mappings $f$ under the weak nondegeneracy assumption of the integrability of $\Delta_f^{-1}$ to some
power $\theta\in(0,1)$.
Corollaries~\ref{c3.4} and~\ref{c4.4} give the Nikolskii--Besov fractional smoothness
of distributions under such weak assumption which generalizes the results of
\cite{BKZ} about the polynomial mappings. Our results also give an estimate
of the total variation distance between two such distributions under a
common weak nondegeneracy assumption
in terms of the
Kantorovich--Rubinstein distance between these distributions.

\section{Definitions and notations}
In this section we introduce the definitions and notation used
throughout the paper.

Let $C_0^\infty(\mathbb{R}^n)$ denote the space
of all infinitely smooth functions with compact support and
let $C_b^\infty(\mathbb{R}^n)$ denote the space of all bounded
smooth functions with bounded derivatives of every order.
The standard Euclidian inner product on $\mathbb{R}^k$ is denoted by
$\langle\cdot, \cdot\rangle$, and the standard norm is denoted by $|\cdot|$.
For the standard Lebesgue measure on $\mathbb{R}^k$ we will use the symbol $\lambda^k$.

Let $\mu$ be a bounded measure on a measurable space.
Recall that $\mu\circ f^{-1}$ denotes the image of the measure $\mu$ under a
$\mu$-measurable mapping $f$,
i.e., the following equality holds:
$$
\mu\circ f^{-1} (A) = \mu\bigl(f^{-1}(A) \bigr).
$$
For a Borel measure $\mu$ on $\mathbb{R}^k$, its shift to the vector $h$ is the measure $\mu_h$
defined
by the equality
$$
\mu_h(A) = \mu(A-h)\quad \text{for every Borel set } A.
$$
The total variation norm of a
Borel measure $\mu$ on $\mathbb{R}^k$ (possibly signed)
is defined by the equality
$$
\|\mu\|_{\rm TV}  := \sup\biggl\{\int \varphi \, d\mu, \ \varphi\in C_0^\infty(\mathbb{R}^k),
\ \|\varphi\|_\infty \le1 \biggr\},
$$
where
$$
\|\varphi\|_\infty:= \sup_{x\in \mathbb{R}^k}|\varphi(x)|.
$$
The Kantorovich--Rubinstein norm (which is sometimes called the Fortet--Mourier norm) of a
Borel measure $\mu$ on $\mathbb{R}^k$ is defined by the formula
$$
\|\mu\|_{\rm KR} := \sup\biggl\{\int \varphi\, d\mu:
\varphi\in C_0^\infty(\mathbb{R}^k),\ \|\varphi\|_\infty \le 1,\ \|\nabla\varphi\|_\infty\le1\biggr\}.
$$
We note here that, for probability measures,
convergence in the Kantorovich--Rubinstein norm is equivalent to weak convergence
(convergence in distribution for random variables).
We also introduce the Kantorovich norm
of a measure $\mu$ on $\mathbb{R}^k$
with finite first moment ($\int |x|\, |\mu|(dx)<\infty$) and with $\mu(\mathbb{R}^k)=0$:
$$
\|\mu\|_{\rm K} := \sup\Bigl\{\int \varphi \,d\mu,
\ \varphi\in C_0^\infty(\mathbb{R}^k),\ \|\nabla\varphi\|_\infty \le1\Bigr\}.
$$

We recall  (see \cite{BIN}, \cite{Nikol77}, and \cite{Stein})
that the Nikolskii--Besov space $B^\alpha(\mathbb{R}^k):=B^\alpha_{1,\infty}(\mathbb{R}^k)$
with $\alpha\in (0,1)$
consists of all functions $\rho\in L^1(\mathbb{R}^k)$
for which there is a constant $C$ such that for every $h\in \mathbb{R}^k$ one has
$$
\int_{\mathbb{R}^k} |\rho(x+h)-\rho(x)|\, dx \le  C|h|^\alpha.
$$
When the function $\rho$ is the density (with respect to $\lambda^k$) of the measure $\mu$
the above condition can be represented in the following form:
$$
\|\mu_h - \mu\|_{\rm TV}\le C|h|^\alpha.
$$

\vskip .1in

We now recall several facts about Gaussian measures on locally convex spaces.

Let $E$ be a locally convex space
with the topological dual $E^*$.
Let $\gamma$
be a centered Gaussian measure on $E$, i.e. it is a Radon measure such that
every functional $\ell\in E^*$ is a normally distributed random variable with zero mean
(its distribution is either the Dirac measure at zero or has a centered Gaussian density).
Let $H\subset E$ be the Cameron--Martin space of the measure $\gamma$
consisting of all vectors $h$ with finite Cameron--Martin norm $|h|_H<\infty$,
where
$$
|h|_H=\sup \biggl\{ \ell(h)\colon\, \int_E \ell^2\, d\gamma \le 1, \ \ell\in E^{*}\biggr\}.
$$
For the standard Gaussian measure on $\mathbb{R}^n$,
the Cameron--Martin space is $\mathbb{R}^n$ itself.
For a general Radon Gaussian measure, the Cameron--Martin space
is a separable Hilbert space (see \cite[Theorem 3.2.7 and Proposition 2.4.6]{Gaus})
with the inner product $\langle\cdot,\cdot\rangle_H$ generated by $|\cdot|_H$.

It is known (see, for example, \cite[Section 2.10]{Gaus})
that for an arbitrary orthonormal family $\{\ell_i\}_{i=1}^n\subset E^*$
in $L^2(\gamma)$
there is an orthonormal family $\{e_i\}_{i=1}^\infty$ in $H$
such that $\ell_i(e_j) = \delta_{i,j}$.
Let $\gamma_n$ be the distribution of the vector $(\ell_1, \ldots, \ell_n)$ on $\mathbb{R}^n$.
This distribution is the standard Gaussian
measure on $\mathbb{R}^n$ with density $(2\pi)^{-n/2}\exp(-|x|^2/2)$.

For a function $f\in L^p(\gamma)$
we set
$$
\|f\|_p := \|f\|_{L^p(\gamma)} := \biggl(\int |f(x)|^p\gamma(dx)\biggr)^{1/p},\quad p\in[1,\infty).
$$

Let $\mathcal{FC}^\infty(E)$ be the set of all functions
$\varphi$ of the form $\varphi(x) = \psi(\ell_1(x), \ldots, \ell_n(x))$,
where $\psi\in C_b^\infty(\mathbb{R}^n)$ and $n\in\mathbb{N}$.

For a function $\varphi\in\mathcal{FC}^\infty(E)$ of the form $\varphi(x) = \psi(\ell_1(x), \ldots, \ell_n(x))$ set
$$
D^1\varphi(x) = \nabla\varphi(x) = \sum_{j=1}^n(\partial_j\psi) (\ell_1(x), \ldots, \ell_n(x)) e_j,
$$
$$
\bigl(D^2\varphi\bigr)_{i,j}(x) =  (\partial_i\partial_j\psi)(\ell_1(x), \ldots, \ell_n(x)).
$$
The Sobolev space $W^{p,m}(\gamma)$, $m\in\{1,2\}$, is the closure of the class $\mathcal{FC}^\infty(E)$ with respect to the norm
$$
\|\varphi\|_{W^{p,m}(\gamma)}:= \|\varphi\|_p + \sum_{i=1}^m\|D^{i}\varphi\|_p,
$$
where $\|D^1\varphi\|_p:=\| |\nabla\varphi|_H\|_p$,
$\|D^2\varphi\|_p:=\| |D^2\varphi|_{HS}\|_p$, and $|\cdot|_{HS}$ is the Hilbert–Schmidt norm.

Let $L$ be the Ornstein--Uhlenbeck operator defined by
$$
L\varphi(x)=\Delta \varphi(x)-\langle x,\nabla \varphi(x)\rangle
$$
for $\varphi\in C_b^\infty(\mathbb{R}^n)$,
where $\Delta$ is the Laplace operator.
We note that
$$
\|L\varphi\|_{L^p(\gamma_n)}\le c_1(p)\|\varphi\|_{W^{p,2}(\gamma_n)}
$$
for $p>1$
with some constant $c_1(p)$ depending only on $p$ (see \cite[Theorem 5.7.1]{Gaus}).

Let $f\colon E\to\mathbb{R}^k$ be a  mapping such that
its components $f_1,\ldots,f_k$ belongs to $W^{1,1}(\gamma)$.
Let us define the Malliavin matrix $M_f$ of the mapping $f$ by
$$
M_f(x)=(m_{i,j}(x))_{i,j\le k},
\quad
m_{i,j}(x):=\langle\nabla f_i(x), \nabla f_j(x)\rangle_H.
$$
Let
$$
A_f:=\{a_{i, j}\}
$$
be the adjugate matrix of $M_f$, i.e., $a_{i, j}=M^{j, i}$, where $M^{j, i}$
is the cofactor of $m_{j, i}$ in the matrix~$M_f$.
Set
$$
\Delta_f:=\det M_f.
$$
Note that
\begin{equation}\label{inver}
\Delta_f\cdot M_f^{-1} = A_f.
\end{equation}

For a function $g\ge0$ we set
\begin{equation*}\label{def-u}
u_\gamma(g, \varepsilon):=\int_0^\infty (s+1)^{-2}\gamma\bigl(g \le \varepsilon s\bigr)\, ds.
\end{equation*}

We need the following simple lemma.
\begin{lemma}\label{lem4.1}
For a function $g\ge0$ and arbitrary numbers $r\ge1, \varepsilon>0$
one has
$$
\int (g + \varepsilon)^{-r}\, d\gamma\le
r\varepsilon^{-r}u_\gamma(g, \varepsilon).
$$
\end{lemma}

\begin{proof}
By Fubini's theorem and Chebyshev's inequality one has
\begin{multline*}
\int (g+\varepsilon)^{-r}\, d\gamma=
r\int_0^{\varepsilon^{-1}}t^{r-1}\gamma\bigl((g+\varepsilon)^{-1}\ge t\bigr)\, dt
\\
=
r\int_0^\infty (s+\varepsilon)^{-r-1}\gamma\bigl(g \le s\bigr)\, ds
\le
r\varepsilon^{-r}\int_0^\infty (s+1)^{-r-1}\gamma\bigl(g \le \varepsilon s\bigr)\, ds
\\
\le
r\varepsilon^{-r}\int_0^\infty (s+1)^{-2}\gamma\bigl(g \le \varepsilon s\bigr)\, ds.
\end{multline*}
The lemma is proved.
\end{proof}

\section{Smoothness properties of measures on $\mathbb{R}^k$}

The following modulus of continuity plays a crucial role below.

\begin{definition}\label{D2.1}
For a measure $\mu$ on $\mathbb{R}^k$ and $t>0$ we set
$$
\sigma(\mu, t):= \sup\Bigl\{\int\partial_e\varphi d\mu:\, \|\varphi\|_\infty\le t,\, \|\partial_e\varphi\|_\infty\le1\Bigr\},
$$
where the supremum is taken over all functions $\varphi\in C_0^\infty(\mathbb{R}^k)$ and over all unit vectors $e$.
\end{definition}

The following theorem is proved in \cite{KosBes}.

\begin{theorem}\label{t2.1}
For any measure $\mu$ on $\mathbb{R}^k$ one has
$$
\|\mu_h - \mu\|_{\rm TV}\le 2\sigma(\mu, |h|/2), \quad \sigma(\mu, t)\le 6k\sup_{|h|\le t}\|\mu_h - \mu\|_{\rm TV}.
$$
\end{theorem}

This theorem implies that the measure $\mu$ is absolutely continuous
with respect to Lebesgue measure
if (and only if) $\sigma(\mu, t)\to0$ as $t\to0$.

The modulus of continuity $\sigma(\mu, \cdot)$ can be used to compare
different distances on the space of probability measures.
In the following theorem we estimate the total variation distance between two probability measures
$\mu$ and $\nu$ in terms of the Kantorovich--Rubinstein distance and the quantity $\sigma(\mu - \nu, \cdot)$.
This result generalizes some estimates from \cite{BKZ} and \cite{Kos}.

\begin{lemma}\label{lem2.1}
Let $\mu$ and $\nu$ be two probability measures on $\mathbb{R}^k$.
Then for any $\varepsilon\in(0,1)$ one has
$$
\|\mu - \nu\|_{\rm TV} \le 3\sqrt{k}\sigma(\mu - \nu, \varepsilon) + \sqrt{k}\varepsilon^{-1}\|\mu - \nu\|_{\rm KR}.
$$
In particular, since $\sigma(\mu - \nu, \varepsilon)\le \sigma(\mu, \varepsilon) + \sigma(\nu, \varepsilon)$,
we have
$$
\|\mu - \nu\|_{\rm TV} \le 6\sqrt{k}\max\{\sigma(\mu, \varepsilon),\sigma(\nu, \varepsilon)\}
+ \sqrt{k}\varepsilon^{-1}\|\mu - \nu\|_{\rm KR}.
$$
\end{lemma}

\begin{proof}
Set
$$
\rho(x) = (2\pi)^{-k/2} \exp(-|x|^2/2)
$$
and
$$
\rho_\varepsilon(x) = \varepsilon^{-k}\rho(t/\varepsilon).
$$
For the measure $\omega := \mu-\nu$ we have
$$
\|\mu - \nu\|_{\rm TV} = \|\omega\|_{\rm TV} \le \|\omega-\omega*\rho_\varepsilon\|_{\rm TV} + \|\omega*\rho_\varepsilon\|_{\rm TV},
$$
where $\omega*\rho_\varepsilon$ is the convolution of the measures $\omega$ and $\rho_\varepsilon\, dx$.
For the first term above, we have
$$
\|\omega-\omega*\rho_\varepsilon\|_{\rm TV}
\le
\int_{\mathbb{R}^k}\|\omega - \omega_{\varepsilon y}\|_{\rm TV}\rho(y)\, dy
\le
2\int_{\mathbb{R}^k}\sigma(\omega, \varepsilon|y|/2)\rho(y)\, dy.
$$
For the second term, we have
\begin{multline*}
\|\omega*\rho_\varepsilon\|_{\rm TV} = \sup_{\substack{\varphi\in C_0^\infty(\mathbb{R}^k)\\ \|\varphi\|_\infty \le1}}
\int_{\mathbb{R}^k} \varphi(x)\int_{\mathbb{R}^k}\rho_\varepsilon(x-y)\, \omega(dy)\, dx
\\
=
\sup_{\substack{\varphi\in C_0^\infty(\mathbb{R}^k)\\ \|\varphi\|_\infty \le1}}
\int_{\mathbb{R}^k}\int_{\mathbb{R}^k}\varphi(x)\rho_\varepsilon(x-y)\, dx\, \omega(dy).
\end{multline*}
Note that
$$
\nabla\int_{\mathbb{R}^k}\varphi(x)\rho_\varepsilon(x-y)\, dx =
\varepsilon^{-k}\int_{\mathbb{R}^k}\varphi(x)\varepsilon^{-1}(\nabla\rho)((x-y)/\varepsilon)\, dx.
$$
Thus,
the Lipschitz constant of the function
$$
y\mapsto\int_{\mathbb{R}^n}\varphi(x)\rho_\varepsilon(x-y)\, dx
$$
can be estimated from above by
$\varepsilon^{-1}\int_{\mathbb{R}^n}|\nabla\rho(x)|\, dx\le \varepsilon^{-1}\sqrt{k}$.
Moreover,
$$
\Bigl|\int_{\mathbb{R}^n}\varphi(x)\rho_\varepsilon(x-y)\, dx\Bigr|\le1\le\varepsilon^{-1}\sqrt{k}
$$
for $\varepsilon\in(0,1)$.
So,
\begin{multline*}
\|\omega\|_{\rm TV} \le 2\int_{\mathbb{R}^k}\sigma(\omega, \varepsilon|y|/2)\rho(y)\, dy + \sqrt{k}\varepsilon^{-1}\|\omega\|_{\rm KR}
\\
=
2\int_{|y|\le2}\sigma(\omega, \varepsilon|y|/2)\rho(y)\, dy + 2\int_{|y|>2}\sigma(\omega, \varepsilon|y|/2)\rho(y)\, dy + \sqrt{k}\varepsilon^{-1}\|f\|_{\rm KR}.
\end{multline*}
In the first integral $\sigma(\omega, \varepsilon|y|/2)\le \sigma(\omega, \varepsilon)$
by monotonicity of the function $\sigma(\omega, \cdot)$
and in the second integral $\sigma(\omega, \varepsilon|y|/2)\le |y|/2 \sigma(\omega, \varepsilon)$,
since $\sigma(\mu, t\varepsilon)\le t\sigma(\mu, \varepsilon)$ for $t\ge1$.
Thus,
$$
\|\omega\|_{\rm TV} \le c_n\sigma(\omega, \varepsilon) + \sqrt{k}\varepsilon^{-1}\|f\|_{\rm KR},
$$
where $c_n = 2\int_{|y|\le2}\rho(y)dy + \int_{|y|>2}|y|\rho(y)dy\le 2+ \sqrt{k}\le 3\sqrt{k}$.
The lemma is proved.
\end{proof}

\begin{remark}
{\rm By a similar reasoning, one can prove that, for an arbitrary pair of probability measures
$\mu$ and $\nu$ on $\mathbb{R}^k$ and any $\varepsilon>0$, one has
$$
\|\mu - \nu\|_{\rm TV} \le 3\sqrt{k}\sigma(\mu - \nu, \varepsilon) + \sqrt{k}\varepsilon^{-1}\|\mu - \nu\|_{\rm K}.
$$
}
\end{remark}

\section{One-dimensional case}
In this section we study smoothness properties of the distribution
$\gamma\circ f^{-1}$ on the real line generated by a Sobolev smooth function $f$ on a locally convex space
equipped with a centered Gaussian measure $\gamma$.

We start with the following technical lemma.

\begin{lemma}\label{lem3.1}
Let $p>1$, $r\ge1$, $a>0$.
Then there is a constant $c(p)$ depending only on $p$ such that
for every function $f\in W^{p,2}(\gamma)$ with
$$
\|f\|_{W^{p,2}(\gamma)}\le a,
$$
and
for every function $g\in W^{r,1}(\gamma)\cap L^\infty(\gamma)$
one has
$$
\int \frac{\langle\nabla g, \nabla f\rangle_H }{|\nabla f|^2_H+\varepsilon^2}\, d\gamma
\le
c(p)a\|g\|_\infty\varepsilon^{-2}u_\gamma(|\nabla f|_H, \varepsilon)^{1-1/p}
$$
for any $\varepsilon>0$.
\end{lemma}

\begin{proof}
We first assume that the functions
$g, f$ belong to $\mathcal{FC}^\infty(E)$ and are of the form
$g = g(\ell_1, \ldots, \ell_n)$,
$f = f(\ell_1, \ldots, \ell_n)$.
Integrating by parts, we have
\begin{multline} \label{est3.1}
\int_E \frac{\langle\nabla g, \nabla f\rangle_H} {|\nabla f|^2_H+\varepsilon^2}\, d\gamma
=
\int_{\mathbb{R}^n} \frac{\langle\nabla g, \nabla f\rangle } {\langle\nabla f,\nabla f\rangle
+\varepsilon^2}\, d\gamma_n
\\
=-\int_{\mathbb{R}^n} g \Bigl(\frac{ Lf}{\langle\nabla f,\nabla f\rangle
+\varepsilon^2}
-2\frac{\langle D^2 f\cdot \nabla f,\nabla f\rangle}{(\langle\nabla f,\nabla f\rangle+\varepsilon^2)^2}
\Bigr)\, d\gamma_n
\\
\le
\|g\|_\infty\int_{\mathbb{R}^n} \frac{| Lf|}{\langle\nabla f,
\nabla f\rangle+\varepsilon^2} \, d\gamma_n +
2\int_{\mathbb{R}^n} \frac{\|D^2 f\|_{HS}}{\langle\nabla f,\nabla f\rangle+\varepsilon^2}\, d\gamma_n
\\
\le
\|g\|_\infty\bigl(\|Lf\|_p
+2\|D^2 f\|_p\bigr)
\|\bigl(|\nabla f|_H^2+\varepsilon^2\bigr)^{-1}\|_{p/(p-1)}
\\
\le
(c_1(p)+2)\|g\|_\infty\|f\|_{W^{p,2}(\gamma)}\|\bigl(|\nabla f|_H^2+\varepsilon^2\bigr)^{-1}\|_{p/(p-1)},
\end{multline}
where $L$ is the Ornstein--Uhlenbeck operator associated with the
standard Gaussian measure~$\gamma_n$.

For a general function $f\in W^{p,2}(\gamma)$, we can take a sequence
$f^n\in \mathcal{FC}^\infty(E)$ such that
$f^n\to f$ in $W^{p,2}(\gamma)$ which
also converges almost everywhere along with first and second derivatives.
Passing to the limit in the above inequality we obtain the same inequality for a general function
$f\in W^{p,2}(\gamma)$
and a function $g\in \mathcal{FC}^\infty(E)$.
Now, for a function $g\in W^{r,1}(\gamma)\cap L^\infty(\gamma)$ we can take
functions $g^n\in\mathcal{FC}^\infty(E)$ such that
$g^n\to g$ in $W^{r,1}(\gamma)$ and almost everywhere.
Let us consider function $\varphi\in C_0^\infty(\mathbb{R})$ such that
$\varphi(t) = t$ for $t\in [-\|g\|_\infty, \|g\|_\infty]$ and $|\varphi(t)|\le 2\|g\|_\infty$.
Then the sequence $\{\varphi(g^n)\}$ also converges to the function $g$ in $W^{r,1}(\gamma)$
and almost everywhere, $\|\varphi(g^n)\|_\infty\le 2\|g\|_\infty$. We can pass
to the limit in the above inequality
and obtain a similar estimate for general functions
$f\in W^{p,2}(\gamma)$ and
$g\in W^{r,1}(\gamma)\cap L^\infty(\gamma)$:
$$
\int \frac{\langle\nabla g, \nabla f\rangle_H} {|\nabla f|^2_H+\varepsilon^2}\, d\gamma
\le
c_2(p)\|g\|_\infty\|f\|_{W^{p,2}(\gamma)}\|(|\nabla f|_H^2+\varepsilon^2\bigr)^{-1}\|_{p/(p-1)}.
$$
By Lemma \ref{lem4.1} we have
$$
\|(|\nabla f|_H^2+\varepsilon^2)^{-1}\|_{p/(p-1)}
\le
2\|(|\nabla f|_H+\varepsilon)^{-2}\|_{p/(p-1)}
\le
c_3(p)\varepsilon^{-2}u_\gamma(|\nabla f|_H, \varepsilon)^{1-1/p}.
$$
Thus,
$$
\int \frac{\langle\nabla u, \nabla f\rangle_H} {|\nabla f|^2_H+\varepsilon^2}\, d\gamma
\le
c(p)a\|g\|_\infty\varepsilon^{-2}u_\gamma(|\nabla f|_H, \varepsilon)^{1-1/p}.
$$
The lemma is proved.
\end{proof}

\begin{theorem}\label{t3.1}
Let $p>1$, $a>0$.
Then there is a constant $c(p)$,
depending only on $p$, such that
for every function $f\in W^{p,2}(\gamma)$ with
$$
\|f\|_{W^{p,2}(\gamma)}\le a
$$
one has
$$
\sigma(\gamma\circ f^{-1}, t)
\le
c(p)a\, t\varepsilon^{-2}u_\gamma(|\nabla f|_H, \varepsilon)^{1-1/p}
+
4 u_\gamma(|\nabla f|_H, \varepsilon)
$$
for every number $\varepsilon>0$.
\end{theorem}
\begin{proof}
For
all $\varphi\in C_0^\infty(\mathbb{R})$ and
$\varepsilon>0$, we can write
$$
\int \varphi'(f)\, d\gamma=
\int \varphi'(f)\frac{\langle\nabla f,\nabla f\rangle_H }
{|\nabla f|_H^2+\varepsilon^2}\, d\gamma
+\varepsilon^2\int \frac{\varphi'(f)}{|\nabla f|_H^2+\varepsilon^2}\, d\gamma.
$$
For the first term, by Lemma \ref{lem3.1}, we have
$$
\int \frac{\langle \nabla(\varphi\circ f), \nabla f\rangle_H }{|\nabla f|^2_H+\varepsilon^2}\, d\gamma
\le
c(p)a\|\varphi\|_\infty\varepsilon^{-2}u_\gamma(|\nabla f|_H, \varepsilon)^{1-1/p}.
$$
The second term, by Lemma \ref{lem4.1}, does not exceed
$$
4\|\varphi'\|_\infty u_\gamma(|\nabla f|_H, \varepsilon).
$$
Therefore,
$$
\int \varphi'(f)\, d\gamma
\le
c(p)a\|\varphi\|_\infty\varepsilon^{-2}u_\gamma(|\nabla f|_H, \varepsilon)^{1-1/p}
+
4\|\varphi'\|_\infty u_\gamma(|\nabla f|_H, \varepsilon).
$$
The theorem is proved.
\end{proof}

Since $u_\gamma(|\nabla f|_H, \varepsilon)\le1$, taking
$\varepsilon=\sqrt t$ in the previous theorem,
we obtain the following result.

\begin{corollary}\label{c3.1}
Let $p>1$, $a>0$.
Then there is a constant $c(p)$,
depending only on $p$, such that
for every function $f\in W^{p,2}(\gamma)$ with
$$
\|f\|_{W^{p,2}(\gamma)}\le a
$$
one has
$$
\sigma(\gamma\circ f^{-1}, t)
\le
(c(p)a + 4) u_\gamma\bigl(|\nabla f|_H, \sqrt t\bigr)^{1-1/p}.
$$
\end{corollary}

The following corollary provides
a quantitative bound in the following result from \cite{BZ}:
convergence in distribution
of random variables $f_n$ from a certain Sobolev class
implies convergence in variation under some uniform nondegeneracy assumption
and uniform boundedness of their Sobolev norms.

\begin{corollary}\label{c3.2}
Let $p>1$ and
let $f_n\in W^{p,2}(\gamma)$ be a sequence such that
$$
\sup_n\|f_n\|_{W^{p,2}(\gamma)}= a<\infty, \quad
\delta(\varepsilon):=\sup_n\gamma(|\nabla f_n|_H\le\varepsilon)\to0.
$$
Assume that the sequence of distributions $\gamma\circ f_n^{-1}$
converges weakly to the measure $\nu$
(equivalently $\|\gamma\circ f_n^{-1} - \nu\|_{\rm KR}\to0$).
Then $\|\gamma\circ f_n^{-1} - \nu\|_{\rm TV}\to0$ and there is a constant $C(p,a)$ such that
$$
\|\gamma\circ f_n^{-1} - \nu\|_{\rm TV}
\le C(p,a)\Bigl(\bigl[\delta\bigl(\|\gamma\circ f_n^{-1} - \nu\|_{\rm KR}^{1/8}\bigr)\bigr]^{1-1/p}  +
\|\gamma\circ f_n^{-1} - \nu\|_{\rm KR}^{(p-1)/(8p)}\Bigr).
$$
\end{corollary}

\begin{proof}
By Lemma \ref{lem2.1} and Corollary \ref{c3.1} we have
$$
\|\gamma\circ f_n^{-1} - \gamma\circ f_m^{-1}\|_{\rm TV}
\le 6(c(p)a + 4)\Bigl(\int_0^\infty(s+1)^{-2}\delta(s\sqrt\varepsilon)\, ds\Bigr)^{1-1/p}  +
\varepsilon^{-1}\|\gamma\circ f_n - \gamma\circ f_m^{-1}\|_{\rm KR}.
$$
Passing to the limit as $m\to\infty$,
we obtain a similar estimate with $\nu$ in place of $\gamma\circ f_m^{-1}$.
We now note that
\begin{multline*}
\int_0^\infty(s+1)^{-2}\delta(s\sqrt\varepsilon)\, ds =
\int_0^{\varepsilon^{-1/4}}(s+1)^{-2}\delta(s\sqrt\varepsilon)\, ds
+\int_{\varepsilon^{-1/4}}^\infty(s+1)^{-2}\delta(s\sqrt\varepsilon)\, ds
\\
\le
\delta(\varepsilon^{1/4}) + \frac{\varepsilon^{1/4}}{\varepsilon^{1/4}+1}
\le
\delta(\varepsilon^{1/4}) + \varepsilon^{1/4}.
\end{multline*}
Taking $\varepsilon = \|\gamma\circ f_n^{-1} - \nu\|_{\rm KR}^{1/2}$ we get
$$
\|\gamma\circ f_n^{-1} - \nu\|_{\rm TV}
\le 12(c(p)a + 4)\Bigl(\bigl[\delta\bigl(\|\gamma\circ f_n^{-1} - \nu\|_{\rm KR}^{1/8}\bigr)\bigr]^{1-1/p}  +
\|\gamma\circ f_n^{-1} - \nu\|_{\rm KR}^{1/8-1/8p}\Bigr).
$$
The corollary is proved.
\end{proof}

The following corollary gives the
Nikolskii--Besov smoothness of $\gamma\circ f^{-1}$ under
the assumption of $\gamma$-integrability of $|\nabla f|_H^{-1}$ to some power $\theta\in(0,1)$.
This result generalizes \cite[Theorem 5.1]{BKZ} to the case of general Sobolev functions.

\begin{corollary}\label{c3.4}
Let $p>1$, $a,b>0$, $\theta\in(0, 1)$.
Set $\alpha:=\frac{p\theta}{2p+\theta}$.
There is a constant $C:=C(p, a, b, \theta)$ such that
for every function $f\in W^{p,2}(\gamma)$ with
$$
\|f\|_{W^{p,2}(\gamma)}\le a, \quad
\int |\nabla f|_H^{-\theta} \, d\gamma\le b
$$
one has
$$
\|(\gamma\circ f^{-1})_h-\gamma\circ f^{-1}\|_{\rm TV}\le C|h|^\alpha, \quad \forall h\in\mathbb{R}.
$$
Equivalently, the measure $\gamma\circ f^{-1}$ possesses a density
from the Nikolskii--Besov class $B^\alpha(\mathbb{R})$.
\end{corollary}
\begin{proof}
Under our assumptions, we have
$$
u_\gamma(|\nabla f|_H, \varepsilon):=\int_0^\infty (s+1)^{-2}\gamma\bigl(|\nabla f|_H \le \varepsilon s\bigr)\, ds
\le
\varepsilon^\theta b\int_0^\infty s^\theta(s+1)^{-2}\, ds = c(b,\theta)\varepsilon^\theta.
$$
By Theorem \ref{t3.1} for every $\varepsilon>0$, one has
\begin{multline*}
\sigma(\gamma\circ f^{-1}, t)
\le
c(p)a t\varepsilon^{-2}u_\gamma(|\nabla f|_H, \varepsilon)^{1-1/p}
+
4 u_\gamma(|\nabla f|_H, \varepsilon)
\\
\le
c(p)a\bigl(c(b, \theta)\bigr)^{1-1/p} t\varepsilon^{-2}\varepsilon^{\theta(1-1/p)}
+
4c(b,\theta)\varepsilon^\theta.
\end{multline*}
Taking $\varepsilon = t^{\frac{p}{2p +\theta}}$ and
applying Theorem \ref{t2.1}
we get the desired bound.
\end{proof}

The following corollary generalizes \cite[Theorem 5.2]{BKZ}.

\begin{corollary}\label{c3.5}
Let $p>1$, $a,b>0$, $\theta\in(0, 1)$.
Set $\beta:=\frac{p\theta}{(2+\theta)p+\theta}$.
There is a constant $C_1:=C_1(p, a, b, \theta)$ such that
such that, for every pair of functions $f, g\in W^{p,2}(\gamma)$ with
$$
\|f\|_{W^{p,2}(\gamma)}\le a, \quad \|g\|_{W^{p,2}(\gamma)}\le a,
\quad \int |\nabla f|_H^{-\theta} \, d\gamma\le b, \quad \int |\nabla g|_H^{-\theta} \, d\gamma\le b,
$$
one has
$$
\|\gamma\circ f^{-1} - \gamma\circ g^{-1}\|_{\rm TV}\le
C_1(p, a, b, \theta)\|\gamma\circ f^{-1} - \gamma\circ g^{-1}\|_{\rm KR}^\beta.
$$
\end{corollary}
\begin{proof}
By Lemma \ref{lem2.1} for each $\varepsilon\in(0,1)$ one has
\begin{multline*}
\|\gamma\circ f^{-1} - \gamma\circ g^{-1}\|_{\rm TV}
\le 6\max\{\sigma(\gamma\circ f^{-1}, \varepsilon),\sigma(\gamma\circ g^{-1}, \varepsilon)\}  +
\varepsilon^{-1}\|\gamma\circ f^{-1} - \gamma\circ g^{-1}\|_{\rm KR}
\\
\le
6C(p, a,b, \theta) \varepsilon^\alpha + \varepsilon^{-1}\|\gamma\circ f^{-1} - \gamma\circ g^{-1}\|_{\rm KR},
\end{multline*}
where $\alpha= \frac{p\theta}{2p+\theta}$.
Taking $\varepsilon = \|\gamma\circ f^{-1} - \gamma\circ g^{-1}\|_{\rm KR}^{\frac{1}{1+\alpha}}$
we get the desired bound.
\end{proof}

\section{Multidimensional case}

We now proceed to the case of multidimensional mappings $f=(f_1, \ldots, f_k)\colon E\to \mathbb{R}^k$
and the properties of their distributions $\gamma\circ f^{-1}$ on $\mathbb{R}^k$.

We start with the following analog of Lemma \ref{lem3.1}.

\begin{lemma}\label{lem4.2}
Let $k\in\mathbb{N}$, $p>1$, $q>1$, $r\ge1$, $a>0$.
Then there exists a number  $C_0:=C_0(k,p,q,a)>0$
such that, for every mapping $f = (f_1, \ldots, f_k)\colon\,
E\to\mathbb{R}^k$, where $f_i\in W^{p,2}(\gamma)$ and
$$
\|f\|_{W^{p,2}(\gamma)}:=\max_{i=1,\ldots, k}\bigl(\|f_i\|_{W^{p,2}(\gamma)}\bigr)\le a,
$$
for every pair of functions $u\in W^{r,1}(\gamma)\cap L^\infty(\gamma)$,
$v\in W^{q,1}(\gamma)$ with
$1/q+1/p+1/r=1$, $1-1/q-(2k+1)/p>0$ and for every number $\varepsilon\in(0,1)$,
one has
$$
\int_E \frac{\langle\nabla u, \nabla f_j\rangle_H v}{\Delta_f+\varepsilon}\, d\gamma
\le
C_0\|u\|_\infty\|v\|_{W^{q,1}(\gamma)}\varepsilon^{-2}u_\gamma(\Delta_f, \varepsilon)^{1-1/q-(2k+1)/p}.
$$
\end{lemma}

\begin{proof}
We first assume that the functions
$u, v, f_i$, $i=1,2,\ldots, k$, belong to $\mathcal{FC}^\infty(E)$ and are of the form
$u = u(\ell_1, \ldots, \ell_n)$, $v = v(\ell_1, \ldots, \ell_n)$,
$f_i = f_i(\ell_1, \ldots, \ell_n)$, for $i=1,2,\ldots, k$.
Integrating by parts, we have
\begin{multline} \label{est4.1}
\int_E \frac{\langle\nabla u, \nabla f_j\rangle_H v}{\Delta_f+\varepsilon}\, d\gamma
=
\int_{\mathbb{R}^n} \frac{\langle\nabla u, \nabla f_j\rangle v}{\Delta_f+\varepsilon}\, d\gamma_n
\\
=-\int_{\mathbb{R}^n} u \Bigl(\frac{v Lf_j}
{\Delta_f+\varepsilon} -
\frac{v\langle\nabla f_j, \nabla \Delta_f\rangle}{(\Delta_f+\varepsilon)^{2}}
+\frac{\langle\nabla f_j, \nabla v \rangle}{\Delta_f+\varepsilon}\Bigr)\, d\gamma_n
\\
\le
\|u\|_\infty\int_{\mathbb{R}^n} |v||Lf_j|
(\Delta_f+\varepsilon)^{-1}\, d\gamma_n
+\|u\|_\infty\int_{\mathbb{R}^n} |v||\langle\nabla f_j, \nabla \Delta_f\rangle|
(\Delta_f+\varepsilon)^{-2}\, d\gamma_n
\\
+
\|u\|_\infty\int_{\mathbb{R}^n} |\langle\nabla f_j, \nabla v \rangle|
(\Delta_f+\varepsilon)^{-1}\, d\gamma_n,
\end{multline}
where $L$ is the Ornstein--Uhlenbeck operator associated with the
standard Gaussian measure~$\gamma_n$.
We now estimate each of these three terms.
The first term in (\ref{est4.1}) can be estimated from above~by
$$
\|u\|_\infty\|v\|_q\|Lf_j\|_p\|(\Delta_f+\varepsilon)^{-1}\|_{\frac{1}{1-1/p-1/q}}
\le
c_1(p)\|u\|_\infty\|v\|_{W^{q,1}(\gamma)}\|f_j\|_{W^{p,2}(\gamma)}\|(\Delta_f+\varepsilon)^{-1}\|_{\frac{1}{1-1/p-1/q}}.
$$
The third term in (\ref{est4.1}) can be estimated by
$$
\|u\|_\infty \|\nabla f_j\|_p\|\nabla v\|_q\|(\Delta_f+\varepsilon)^{-1}\|_{\frac{1}{1-1/p-1/q}}
\le
\|u\|_\infty \|f_j\|_{W^{p,2}(\gamma)}\|v\|_{W^{q,1}(\gamma)}\|(\Delta_f+\varepsilon)^{-1}\|_{\frac{1}{1-1/p-1/q}},
$$

To estimate the second term in (\ref{est4.1}) we need
to estimate the gradient of the determinant. We note that for an
arbitrary matrix $C$, one has
$|\det C|\le\prod_i |c^i|$, where $\{c^i\}$ are columns of the matrix $C$. We have
$\langle\nabla f_j, \nabla \Delta_f\rangle = \sum_i \det C_i$, where $C_i=\{c_i^{m,r}\}$ is the matrix such that
$c_i^{m,r} = \langle\nabla f_m, \nabla f_r\rangle$ for $r\ne i$ and
$c_i^{m,i} = \langle D^2f_m\cdot\nabla f_i, \nabla f_j\rangle + \langle D^2f_i\cdot\nabla f_m, \nabla f_j\rangle$.
Thus,
\begin{multline*}
|\langle\nabla f_j, \nabla \Delta_f\rangle| \le \sum_i |\det C_i|\le
\sum_i \prod_r |c^r_i|\\
\le \Bigl(\sum_m|\nabla f_m| \Bigr)^{2(k-1)}\sum_i\sum_m|\langle D^2f_m\cdot\nabla f_i, \nabla f_j\rangle|
+ |\langle D^2f_i\cdot\nabla f_m, \nabla f_j\rangle|
\\
\le
2|\nabla f_j|\Bigl(\sum_m|\nabla f_m| \Bigr)^{2k-1}\sum_i\|D^2f_i\|_{HS}
\le
2\Bigl(\sum_m|\nabla f_m| \Bigr)^{2k}\sum_i\|D^2f_i\|_{HS}.
\end{multline*}
So, the second term in (\ref{est4.1}) is estimated by
\begin{multline*}
2\|u\|_\infty
\int_{\mathbb{R}^n} |v|\Bigl(\sum_m|\nabla f_m| \Bigr)^{2k}\Bigl(\sum_i\|D^2f_i\|_{HS}\Bigr)
(\Delta_f+\varepsilon)^{-2}\, d\gamma_n
\\
\le
2\|u\|_\infty\|v\|_q\Bigl\|\sum_m|\nabla f_m| \Bigr\|_p^{2k}\Bigl\|\sum_i\|D^2f_i\|_{HS}\Bigr\|_p
\|(\Delta_f+\varepsilon)^{-2}\|_{\frac{1}{1-1/q-(2k+1)/p}}
\\
\le
c_2(k)\|u\|_\infty\|v\|_{W^{q,1}(\gamma)}\|f\|_{W^{p,2}(\gamma)}^{2k+1}
\|(\Delta_f+\varepsilon)^{-2}\|_{\frac{1}{1-1/q-(2k+1)/p}}
\end{multline*}
for some constant $c_2(k)$, which depends only on $k$.

Therefore, we have
\begin{multline*}
\int \frac{\langle\nabla u, \nabla f_j\rangle v}{\Delta_f+\varepsilon}\, d\gamma
\le
\bigl(c_1(p)+1\bigr)\|u\|_\infty\|v\|_{W^{q,1}(\gamma)}\|f_j\|_{W^{p,2}(\gamma)}\|(\Delta_f+\varepsilon)^{-1}\|_{\frac{1}{1-1/p-1/q}}
\\
+
c_2(k)\|u\|_\infty\|v\|_{W^{q,1}(\gamma)}\|f\|_{W^{p,2}(\gamma)}^{2k+1}
\|(\Delta_f+\varepsilon)^{-2}\|_{\frac{1}{1-1/q-(2k+1)/p}}
\end{multline*}
for functions
$u, v, f_i \in \mathcal{FC}^\infty(E)$, $i=1,2,\ldots, k$.
For general functions $f_i\in W^{p,2}(\gamma)$, $v\in W^{q,1}(\gamma)$, we can take sequences
$f_i^n\in \mathcal{FC}^\infty(E)$, $v^n\in \mathcal{FC}^\infty(E)$ such that
$f_i^n\to f_i$ in $W^{p,2}(\gamma)$, $v^n\to v$ in $W^{q,1}(\gamma)$ and both sequences
(along with the sequences of their derivatives)
also converge almost everywhere.
Passing to the limit in the above inequality we obtain the same inequality for general functions
$f_i\in W^{p,2}(\gamma)$, $i=1,2,\ldots, k$, $v\in W^{q,1}(\gamma)$,
and functions $u\in \mathcal{FC}^\infty(E)$.
Now, for a function $u\in W^{r,1}(\gamma)\cap L^\infty(\gamma)$, we can take
functions $u^n\in\mathcal{FC}^\infty(E)$ such that
$u^n\to u$ in $W^{r,1}(\gamma)$ and almost everywhere.
Let us consider a function $\varphi\in C_0^\infty(\mathbb{R})$ such that
$\varphi(t) = t$ for $t\in [-\|u\|_\infty, \|u\|_\infty]$ and $|\varphi(t)|\le 2\|u\|_\infty$.
Then, the sequence $\{\varphi(u^n)\}$ also converges to the function $u$ in $W^{r,1}(\gamma)$
and almost everywhere, $\|\varphi(u^n)\|_\infty\le 2\|u\|_\infty$. We can pass
to the limit in the above inequality
and obtain a similar estimate for general functions
$f_i\in W^{p,2}(\gamma)$, $i=1,2,\ldots, k$, $v\in W^{q,1}(\gamma)$,
$u\in W^{r,1}(\gamma)\cap L^\infty(\gamma)$:
\begin{multline*}
\int_E \frac{\langle\nabla u, \nabla f_j\rangle_H v}{\Delta_f+\varepsilon}\, d\gamma
\le
c_3(k,p)\|u\|_\infty\|v\|_{W^{q,1}(\gamma)}\bigl(\|f_j\|_{W^{p,2}(\gamma)}\|(\Delta_f+\varepsilon)^{-1}\|_{\frac{1}{1-1/p-1/q}}
\\
+
\|f\|_{W^{p,2}(\gamma)}^{2k+1}
\|(\Delta_f+\varepsilon)^{-2}\|_{\frac{1}{1-1/q-(2k+1)/p}}\bigr),
\end{multline*}
with $c_3(k, p) = 2\bigl(c_1(p)+1\bigr)+ 2c_2(k)$.

By Lemma \ref{lem4.1} we have
$$
\|(\Delta_f+\varepsilon)^{-1}\|_{\frac{1}{1-1/p-1/q}}
\le
2\varepsilon^{-1}u_\gamma(\Delta_f, \varepsilon)^{1-1/p-1/q}.
$$
and
$$
\|(\Delta_f+\varepsilon)^{-2}\|_{\frac{1}{1-1/q-(2k+1)/p}}
\le
3\varepsilon^{-2}u_\gamma(\Delta_f, \varepsilon)^{1-1/q-(2k+1)/p}.
$$
Since $\varepsilon\le1$ and $u_\gamma(\Delta_f, \varepsilon)\le1$
we have
$$
\varepsilon^{-1}u_\gamma(\Delta_f, \varepsilon)^{1-1/p-1/q}\le\varepsilon^{-2}u_\gamma(\Delta_f, \varepsilon)^{1-1/q-(2k+1)/p}.
$$
Thus,
$$
\int_E \frac{\langle\nabla u, \nabla f_j\rangle_H v}{\Delta_f+\varepsilon}\, d\gamma
\le
C_0(k,p,q, a)\|u\|_\infty\|v\|_{W^{q,1}(\gamma)}\varepsilon^{-2}u_\gamma(\Delta_f, \varepsilon)^{1-1/q-(2k+1)/p}
$$
with
$C_0(k,p,q, a)
=
c_3(k,p)(2a+3a^{2k+1})$.
The lemma is proved.
\end{proof}

\begin{theorem}\label{t4.1}
Let $k\in\mathbb{N}$, $a>0$, and $p>4k-1$. Then there exists a number  $C_1:=C_1(p, k, a)>0$
such that, for every mapping $f = (f_1, \ldots, f_k)\colon\,
E\to\mathbb{R}^k$, where $f_i\in W^{p,2}(\gamma)$ and
$$
\|f\|_{W^{p,2}(\gamma)}:=\max_{i=1,\ldots, k}\bigl(\|f_i\|_{W^{p,2}(\gamma)}\bigr)\le a,
$$
for every $\varepsilon\in(0,1)$, one has
$$
\sigma(\gamma\circ f^{-1}, t)
\le
C_1t\varepsilon^{-2}u_\gamma(\Delta_f, \varepsilon)^{1-(4k-1)/p}
+
u_\gamma(\Delta_f, \varepsilon).
$$
\end{theorem}
\begin{proof}
Fix an arbitrary function $\varphi\in C_0^\infty(\mathbb{R}^k)$ with $\|\varphi\|_\infty\le t$,
$\|\partial_e\varphi\|_\infty\le1$,
and an arbitrary unit vector $e\in \mathbb{R}^k$.
It can be easily verified that
$$
M_f (\partial_1\varphi(f),\ldots, \partial_k\varphi(f)) =
\bigl(\langle\nabla (\varphi\circ f), \nabla f_1\rangle_H, \ldots,
\langle\nabla (\varphi\circ f), \nabla f_k\rangle_H\bigr).
$$
Here the left-hand side is interpreted as the standard product of
a matrix and a column vector.
Then by (\ref{inver}) we have
$$
(\partial_e\varphi)(f) \Delta_f =
\bigl\langle v , A_f e\bigr\rangle, \quad
v=\bigl(\langle\nabla (\varphi\circ f), \nabla f_1\rangle_H, \ldots,
\langle\nabla (\varphi\circ f), \nabla f_k\rangle_H\bigr)
$$
which yields the following equality:
$$
\Delta_f(\partial_e\varphi)(f) = \sum_{i,j}\langle\nabla (\varphi\circ f), \nabla f_j\rangle_H a^{i,j}_fe_i.
$$
For any fixed number $\varepsilon\in (0,1)$ we can write
\begin{equation}\label{ek4.1}
\int \partial_e\varphi(f)\, d\gamma
=
\int \partial_e\varphi(f)\frac{\Delta_f}{\Delta_f+\varepsilon}\, d\gamma+
\varepsilon\int \partial_e\varphi(f)(\Delta_f+\varepsilon)^{-1}\, d\gamma.
\end{equation}
For the first term by the above reasoning we have
$$
\int \partial_e\varphi\frac{\Delta_f}{\Delta_f+\varepsilon}\, d\gamma
=
\sum_{i,j}\int \frac{\langle\nabla (\varphi\circ f), \nabla f_j\rangle_H a^{i,j}_fe_i}
{\Delta_f+\varepsilon}\, d\gamma.
$$
We note that $a^{i,j}_fe_i\in W^{p/(2k-2),1}(\gamma)$ and there is a constant
$c_4(k)$ such that
$$
\|a^{i,j}_fe_i\|_{W^{p/(2k-2),1}(\gamma)}\le c_4(k)\|f\|_{W^{p,2}(\gamma)}^{2k-2}\le c_4(k)a^{2k-2}.
$$
We also note that $\varphi\circ f\in W^{p,1}(\gamma)$ and $(2k-2)/p + 1/p + 1/p\le1$. Hence
$\varphi\circ f\in W^{\frac{1}{1-(2k-1)/p},1}(\gamma)$ and
$\|\varphi\circ f\|_{W^{\frac{1}{1-(2k-1)/p},1}(\gamma)}\le \|\varphi\circ f\|_{W^{p,1}(\gamma)}$.
Moreover, we have
$$1-(2k-2)/p-(2k+1)/p = 1 - (4k-1)/p>0.
$$
Applying now Lemma \ref{lem4.2} with $r= (1- (2k-1)/p)^{-1}$ and $q = p/(2k-2)$ we obtain
$$
\int \partial_e\varphi\frac{\Delta_f}{\Delta_f+\varepsilon}\, d\gamma
\le
C_1(k, p, a)\|\varphi\|_\infty \varepsilon^{-2}u_\gamma(\Delta_f, \varepsilon)^{1-(4k-1)/p},
$$
with
$C_1(k,p, a) = k^2c_4(k)C_0(k,p,p/(2k-2), a)a^{2k-2}$.

Using Lemma \ref{lem4.1}, we can estimate the second term in (\ref{ek4.1}) in the following way:
$$
\varepsilon\int \partial_e\varphi(f)(\Delta_f+\varepsilon)^{-1}\, d\gamma \le
\|\partial_e\varphi\|_\infty u_\gamma(\Delta_f, \varepsilon)\le u_\gamma(\Delta_f, \varepsilon).
$$
Hence we have obtained the estimate
$$
\int \partial_e\varphi\, d\gamma \le
C_1(k, p, a)\|\varphi\|_\infty \varepsilon^{-2}u_\gamma(\Delta_f, \varepsilon)^{1-(4k-1)/p}
+
u_\gamma(\Delta_f, \varepsilon).
$$
Since $\|\varphi\|_\infty\le t$, the theorem is proved.
\end{proof}

Taking $\varepsilon=\sqrt{t}$ we get the following result.

\begin{corollary}\label{c4.1}
Let $k\in\mathbb{N}$, $a>0$, and $p>4k-1$. Then there exists a constant  $C:=C(p, k, a)>0$
such that, for every mapping $f = (f_1, \ldots, f_k)\colon\,
E\to\mathbb{R}^k$, where $f_i\in W^{p,2}(\gamma)$ and
$$
\|f\|_{W^{p,2}(\gamma)}:=\max_{i=1,\ldots, k}\bigl(\|f_i\|_{W^{p,2}(\gamma)}\bigr)\le a,
$$
for every $t\in(0,1)$, one has
$$
\sigma(\gamma\circ f^{-1}, t)
\le
C(p,k,a)u_\gamma(\Delta_f, \sqrt{t})^{1-(4k-1)/p}.
$$
\end{corollary}

The following corollary
is a multidimensional analog of Corollary \ref{c3.2}.
It asserts that convergence in distribution
of random vectors $f_n$ from a Sobolev class
implies convergence in variation provided they are uniformly nondegenerate
and uniformly bounded in the Sobolev norm.

\begin{corollary}\label{c4.2}
Let $k\in\mathbb{N}$, $a>0$, and $p>4k-1$.
Let $f_n=(f_{n,1}, \ldots, f_{n,k})\in W^{p,2}(\gamma)$ be a sequence of mappings such that
$$
\sup_n\|f_n\|_{W^{p,2}(\gamma)}= a<\infty, \quad
\delta(\varepsilon):=\sup_n\gamma(\Delta_{f_n}\le\varepsilon)\to0.
$$
Assume also that the sequence of distributions $\gamma\circ f_n^{-1}$ converges
weakly to some measure $\nu$
(equivalently, $\|\gamma\circ f_n^{-1} - \nu\|_{\rm KR}\to0$).
Then $\|\gamma\circ f_n^{-1} - \nu\|_{\rm TV}\to0$ and
$$
\|\gamma\circ f_n^{-1} - \nu\|_{\rm TV}
\le C_2(p,k,a)\Bigl(\bigl[\delta\bigl(\|\gamma\circ f_n^{-1} - \nu\|_{\rm KR}^{\frac 1 8}\bigr)\bigr]^{1-(4k-1)/p}  +
\|\gamma\circ f_n^{-1} - \nu\|_{\rm KR}^{\frac{1-(4k-1)/p}{8}}\Bigr).
$$
\end{corollary}

\begin{proof}
By Lemma \ref{lem2.1} and Corollary \ref{c4.1} we have
$$
\|\gamma\circ f_n^{-1} - \gamma\circ f_m^{-1}\|_{\rm TV}
\le 6C(p,k,a)\Bigl(\int_0^\infty(s+1)^{-2}\delta(s\sqrt\varepsilon)\, ds\Bigr)^{1-(4k-1)/p}  +
\sqrt{k}\varepsilon^{-1}\|\gamma\circ f_n - \gamma\circ f_m^{-1}\|_{\rm KR}.
$$
Passing to the limit as $m\to\infty$,
we obtain a similar estimate with $\nu$ in place of $\gamma\circ f_m^{-1}$.
Now we proceed as in Corollary \ref{c3.2}:
\begin{multline*}
\int_0^\infty(s+1)^{-2}\delta(s\sqrt\varepsilon)\, ds =
\int_0^{2^{1/8}\varepsilon^{-1/4}}(s+1)^{-2}\delta(s\sqrt\varepsilon)\, ds
+\int_{2^{1/8}\varepsilon^{-1/4}}^\infty(s+1)^{-2}\delta(s\sqrt\varepsilon)\, ds
\\
\le
\delta(2^{1/8}\varepsilon^{1/4}) + \frac{\varepsilon^{1/4}}{\varepsilon^{1/4}+2^{1/8}}
\le
\delta(2^{1/8}\varepsilon^{1/4}) + 2^{-1/8}\varepsilon^{1/4}.
\end{multline*}
Taking $\varepsilon = 2^{-1/2}\|\gamma\circ f_n^{-1} - \nu\|_{\rm KR}^{1/2}\le1$ we get
$$
\|\gamma\circ f_n^{-1} - \nu\|_{\rm TV}
\le C_2(p, k, a)\Bigl(\bigl[\delta\bigl(\|\gamma\circ f_n^{-1} - \nu\|_{\rm KR}^{1/8}\bigr)\bigr]^{1-(4k-1)/p}  +
\|\gamma\circ f_n^{-1} - \nu\|_{\rm KR}^{1/8-(4k-1)/8p}\Bigr).
$$
The corollary is proved.
\end{proof}

We now apply Theorem \ref{t4.1} to show the Nikolskii--Besov
smoothness of $\gamma\circ f^{-1}$ under our
weak nondegeneracy condition: $\Delta_f^{-1}$ is $\gamma$-integrable to some power $\theta\in(0,1)$.
The following corollary generalizes \cite[Theorem 4.1]{BKZ}.

\begin{corollary}\label{c4.4}
Let $k\in\mathbb{N}$, $a>0$, $b>0$, $\theta\in(0, 1)$, $p>4k-1$. Set $\alpha:=\frac{p\theta}{2p+(4k-1)\theta}$.
Then there exists a number  $C:=C(p, k, a, b, \theta)>0$
such that, for every mapping $f = (f_1, \ldots, f_k)\colon E\to\mathbb{R}^k$, where $f_i\in W^{p,2}(\gamma)$ and
$$
\|f\|_{W^{p,2}(\gamma)}:=\max_{i=1,\ldots, k}\bigl(\|f_i\|_{W^{p,2}(\gamma)}\bigr)\le a, \quad
\int \Delta_f^{-\theta} \, d\gamma\le b,
$$
one has
$$
\|(\gamma\circ f^{-1})_h-\gamma\circ f^{-1}\|_{\rm TV}\le C|h|^\alpha \quad \forall h\in\mathbb{R}^k.
$$
In other words, the density of $\gamma\circ f^{-1}$ belongs to the Nikolskii--Besov space~$B^\alpha(\mathbb{R}^k)$.
\end{corollary}
\begin{proof}
Let us estimate $u_\gamma(\Delta_f, \varepsilon)$:
$$
u_\gamma(\Delta_f, \varepsilon):=\int_0^\infty (s+1)^{-2}\gamma\bigl(\Delta_f \le \varepsilon s\bigr)\, ds
\le
\varepsilon^\theta b\int_0^\infty s^\theta(s+1)^{-2}\, ds = c_1(b, \theta)\varepsilon^\theta.
$$
By Theorem \ref{t4.1}
for $\varepsilon\in(0,1)$ one has
\begin{multline*}
\sigma(\gamma\circ f^{-1}, t)
\le
C_1(p,k,a)t\varepsilon^{-2}u_\gamma(\Delta_f, \varepsilon)^{1-(4k-1)/p}
+
u_\gamma(\Delta_f, \varepsilon)
\\
\le
C_2(p, k, a,b, \theta)(t\varepsilon^{-2+(1-(4k-1)/p)\theta} + \varepsilon^\theta).
\end{multline*}
Taking $\varepsilon = t^{\frac{p}{2p +(4k-1)\theta}}$ for $t<1$ and noting that
$\sigma(\gamma\circ f^{-1}, t)\le 1\le t$ for $t\ge1$,
by Theorem \ref{t2.1}
we get the desired bound.
\end{proof}

The next corollary is a generalization of \cite[Theorem 4.2]{BKZ} to the case of Sobolev mappings
in place of polynomials.

\begin{corollary}\label{c4.5}
Let $k\in\mathbb{N}$, $a>0$, $b>0$, $\theta\in(0, 1)$, $p>4k-1$. Set $\alpha:=\frac{p\theta}{2p+(4k-1)\theta}$.
Then there exists a number $C:=C(p, k, a, b, \theta)>0$
such that for every pair of mappings $f = (f_1, \ldots, f_k), g = (g_1, \ldots, g_k)\colon\,
E\to\mathbb{R}^k$, where $f_i, g_i\in W^{p,2}(\gamma)$ and
$$
\|f\|_{W^{p,2}(\gamma)}\le a, \quad \|g\|_{W^{p,2}(\gamma)}\le a,
\quad \int \Delta_f^{-\theta} \, d\gamma\le b, \quad \int \Delta_g^{-\theta} \, d\gamma\le b,
$$
one has
$$
\|\gamma\circ f^{-1} - \gamma\circ g^{-1}\|_{\rm TV}\le
C(p, k, a, b, \theta)\|\gamma\circ f^{-1} - \gamma\circ g^{-1}\|_{\rm KR}^{\frac{\alpha}{1+\alpha}}.
$$
\end{corollary}
\begin{proof}
By Lemma \ref{lem2.1}, for an arbitrary $\varepsilon\in(0,1)$, one has
\begin{multline*}
\|\gamma\circ f^{-1} - \gamma\circ g^{-1}\|_{\rm TV}
\le 6\sqrt{k}\max\{\sigma(\gamma\circ f^{-1}, \varepsilon),\sigma(\gamma\circ g^{-1}, \varepsilon)\}  +
\sqrt{k}\varepsilon^{-1}\|\gamma\circ f^{-1} - \gamma\circ g^{-1}\|_{\rm KR}
\\
\le
C_1(p,k, a,b, \theta) (\varepsilon^\alpha + \varepsilon^{-1}\|\gamma\circ f^{-1} - \gamma\circ g^{-1}\|_{\rm KR}).
\end{multline*}
Taking $\varepsilon = 2^{-1}\|\gamma\circ f^{-1} - \gamma\circ g^{-1}\|_{\rm KR}^{\frac{1}{1+\alpha}}$
we get the desired bound.
\end{proof}

\vskip .1in

The author is a Young
Russian Mathematics award winner and would like to thank its sponsors and jury.

This research was supported by the Russian Science Foundation Grant 17-11-01058
at Lo\-mo\-no\-sov
Moscow State University.

\end{document}